\documentclass[11pt]{article}
\usepackage{geometry}
\usepackage{amssymb}
\usepackage{latexsym}
\usepackage{amsmath}
\usepackage{indentfirst}
\usepackage{graphicx}
\usepackage[colorlinks=true]{hyperref}
\usepackage{mathrsfs}
\usepackage{color}
\usepackage{extarrows}
\usepackage{breqn}

\usepackage{tikz}
\newcommand\crossint{%
\mathrel{\tikz[baseline={(I.base)}]{
\node[inner sep=0](I){$\int$};
\draw[line width=0.5pt]([yshift=1.5ex, xshift=0.4ex]I.south west)--([yshift=1.5ex, xshift=-0.4ex]I.south east);
}}
}
\usepackage{cancel} 

\newtheorem{Thm}{Theorem}[section]

\newtheorem{Lem}{Lemma}[section]
\newtheorem{Pro}{Proposition}[section]

\newtheorem{Def}{Definition}[section]

\newcommand{\R}{\mathbb{R}}

\numberwithin{equation}{section} \numberwithin{figure}{section}

\newenvironment{proof}{\medskip\par\noindent{\bf Proof\/}.\quad}{\qquad
\raisebox{-0.5mm}{\rule{2.5mm}{2.5mm}}\vspace{7pt}}

\begin{document}
\title{Symmetry and monotonicity of singular solutions for the Hartree equation}

\author{Ying Cai$^{1}$,\quad Guangze Gu$^{1,2}$,\quad Aleks Jevnikar$^{2}$ \\
\footnotesize{\em
1. Department  of  Mathematics, Yunnan  Normal  University, Kunming 650500, China. }\\
\footnotesize{\em
2.  Department of Mathematics, Computer Science and Physics, University of Udine, Via delle }\\
\footnotesize{\em
 Scienze 206, 33100 Udine, Italy.}\\
\footnotesize{Email: yingcai0412@163.com, guangzegu@163.com, aleks.jevnikar@uniud.it}
 }

\date{}

\date{} \maketitle
\begin{abstract}
In this paper we are concerned with positive singular solutions of the following nonlocal Hartree equation
$$-\Delta u\!=\Big( \int_{\mathbb{R}^N\setminus \Gamma}\frac{F(u(y))}{|x-y|^\mu}dy \Big)f (u(x)), \quad x\in \mathbb{R}^N\setminus\Gamma,$$
where $F$ is the primitive of $f$ and $\Gamma$ is the singular set.
Under suitable assumptions, we prove that $u$ is symmetric and monotone with respect to the singular set by using moving plane methods. Furthermore, we complement this study by showing the existence, for a model problem, of a singular solution with the desired properties.

\end{abstract}

\vspace{6mm} \noindent{\bf Keywords:} Hartree equation; Singular solutions; Moving plane method.

\vspace{6mm} \noindent

\section{Introduction}

The following nonlocal Choquard equation:
\begin{equation}\label{eq1.1}
-\Delta u+V(x)u=\Big( \int_{\mathbb{R}^N}\frac{F(u(y))}{|x-y|^\mu}dy \Big)f (u(x)),
\end{equation}
has been  studied by many  scholars, where $0<\mu<N$, $V(x)\geq0$ and $F$ is the primitive of $f$.
Equation \eqref{eq1.1} is related to the Choquard equation arising from the study of Bose-Einstein condensation and can be used to depict the finite range many-body interactions between particles.
When $N=3$, $\mu=1$ and $f(u)=u$, equation \eqref{eq1.1}  was proposed by  Choquard \cite{Lieb1967SAM} as an approximation to Hartree-Fock theory for a one component plasma and was also introduced by Pekar \cite{Pekar1954} to describe the the quantum theory of a polaron at rest. See also \cite{Bahrami2014}, \cite{Choquard-Stubbe-Vuffray2008DIE},\cite{Moroz-Van-Schaftingen2013JFA}, \cite{Tod-Moroz1999N} and the references therein for more mathematical and physical background of equation \eqref{eq1.1}.
The appearance of the nonlocal term $\Big(|x|^{-\mu} *F(u)\Big)f(u)$ gives rise to some mathematical difficulties for the study of Choquard equation, which has received increasing attention from many authors.

\par
When $N=3$, $\mu=1$ and $f(u)=u$,  \cite{Ma-Zhao2010ARMA} proved that
the positive solutions of \eqref{eq1.1}  are unique and non-degenerate up to translations.
Later, \cite{Xiang2016CVPDE} generalized the above results to the case $N=3$, $\mu=1$, $f(u)=|u|^{p-2}u$  and  $2 <p< 2 + \delta$ with $\delta$ small enough.
Lenzmann \cite{Lenzmann2009APDE} showed the uniqueness and non-degeneracy of ground states for the pseudorelativistic Hartree equation.
For more results of Choquard equation, we may refer readers  to  \cite{Alves-Yang2014JDE,Alves-Yang2016PRSESA,Cingolani-Clapp-Secchi2012ZAMP, Lieb-Elliott1976SAM, Ma-Zhao2010ARMA,Moroz-VanSchaftingen2013JFA, Li2019arXiv, Li-Yang-Zhou2022SCM,  Deng-Peng-Yang2023JDE, Wang-Chen-Niu2021BBMS} and the references therein.

When $V(x)=0 $, equation \eqref{eq1.1} reduces to the following nonlocal elliptic equations
\begin{equation}\label{eq1.2}
-\Delta u=\Big( \int_{\mathbb{R}^N}\frac{F(u(y))}{|x-y|^\mu}dy \Big)f (u(x)),
\end{equation}
which is also called Hartree equation. For $f(u)=u$, Lei \cite{Lei2013SIAM} showed the nonexistence of positive stable solutions.
Zhao \cite{Zhao2018AMSSB} showed that equation \eqref{eq1.1} has no nontrivial solution with finite Morse index for $f(u)=|u|^{p-2}u$ with $2 <p< 2_\mu^*=\frac{2N-\mu}{N-2}$.
For the critical exponent case $f(u)=|u|^{2_\mu^*-1}$, Du and Yang \cite{Du-Yang2019DCDS} showed the symmetry and uniqueness of the positive solutions by the moving plane method in integral forms.  A similar result can refer to \cite{Guo-Hu-Peng-Shuai2019CVPDE}.
Giacomoni, Wei and Yang \cite{Giacomoni-Wei-Yang2020NA} proved the nondegeneracy of the unique positive solutions for  the  critical Hartree type equations.
Huang and Liu \cite{Huang-Liu2024JFPTA} investigated classification and  non-degeneracy  of positive solutions of the fractional Hartree equation.
For more results concerning the method of moving planes or spheres to Hartree equation involving the fractional Laplacian, please see \cite{Dai-Fang-Qin2018JDE, Dai-Liu-Qi2021SIAM, Le2019N, Ma-Shang-Zhang2020PJM, Peng2022MZ} and the references therein.

Here, we are interested in the symmetry and monotonicity of the solution $u$ of the following equation
\begin{equation}\label{eq1.3}
\left\{\begin{array}{ll}
-\Delta u\!=\Big( \int_{\mathbb{R}^N\setminus\Gamma}\frac{F(u(y))}{|x-y|^\mu}dy \Big)f (u(x)),     \quad  &x\in  \mathbb{R}^N\setminus\Gamma, \\
u>0,                                         &x\in  \mathbb{R}^{N}\setminus\Gamma,\\
\end{array} \right.
\end{equation}
where $\Gamma\subset \{x_1=0\}$ is a singular set which is compact and with zero 2-capacity.

When $\mu \to N$, equation \eqref{eq1.3} reduces to the following local elliptic equation
\begin{equation}\label{eq1.4}
\left\{\begin{array}{ll}
-\Delta u\!=f (u(x)),     \quad  &x\in  \mathbb{R}^N\setminus\Gamma, \\
u>0,                                         &x\in  \mathbb{R}^{N}\setminus\Gamma,
\end{array} \right.
\end{equation}
Sciunzi \cite{Sciunzi2017JMPA} showed the symmetry and monotonicity properties of the positive singular solutions to equation \eqref{eq1.4} by using an improved version of the moving plane method. The classical moving plane method was introduced by Alexandrov \cite{Alexandrov1962AMPA} and subsequently improved by many researchers (see \cite{Berestycki-Nirenberg1991BSBM, Caffarelli-Gidas-Spruck1989CPAM, Chen-Li1991DMJ, Serrin1971ARMA}). The result of Sciunzi was generalized to the case $f(u)=u^{2^*-1}$ by  Esposito, Farina and Sciunzi in \cite{Esposito-Farina-Sciunzi2018JDE}.
Based on the flexibility of the technique developed by Sciunzi in \cite{Sciunzi2017JMPA}, this technique has been applied to the case of the fractional Laplacian \cite{Montoro-Punzo-Sciunzi2018AMPA}, the $p$-Laplacian operator \cite{Esposito-Montoro-Sciunzi2019JMPA, Montoro2019NoDEA},
double phase problems\cite{Biagi-Esposito-Vecchi2021JDE, Biagi-Esposito-Vecchi2023DIE},
cooperative elliptic systems \cite{Biagi-Valdinoci-Vecchi2020PM, Esposito2020DCDS} and mixed local-nonlocal elliptic operators \cite{Biagi-Vecchi-Dipierro-Valdinoci2021PRSES}.

When $f(u)=|u|^p$,  $\Gamma=\{0\}$ and $\mu \to N$, equation \eqref{eq1.3} is just
\begin{equation}\label{eq1.5}
-\Delta u\!=u^p, \quad\quad    x\in \,  \mathbb{R}^N\setminus\{0\}.
\end{equation}
Caffarelli, Gidas and Spruck\cite{Caffarelli-Gidas-Spruck1989CPAM} showed that any solution of \eqref{eq1.5} must be radially symmetric.
The singularities of \eqref{eq1.5}  have been studied in \cite{Lions1980JDE} for $1<p\leq \frac{n+2}{n-2}$, in \cite{Aviles1983IUMJ} for $p=\frac{n}{n-2}$, in \cite{Bidaut-Veron1991IM} for $p> \frac{n+2}{n-2}$,
in\cite{Caffarelli-Gidas-Spruck1989CPAM, Gidas-Spruck1981CPAM} for $\frac{n}{n-2}<p\leq \frac{n+2}{n-2}$, and in \cite{Caju-doo-Santos2019AIHP, Caffarelli-Gidas-Spruck1989CPAM,  Korevaar-Mazzeo-Pacard-Schoen1999IM} for $p=\frac{n+2}{n-2}$.

On the other hand, the study of equation like \eqref{eq1.2} or \eqref{eq1.3} is related to the  well-known Hardy-Littlewood-Sobolev (HLS for short) inequality, see \cite{Lieb1983AM, Lieb-Loss2001book}. For the sake of  convenience we recall the HLS inequality:
\begin{Pro}\label{Pro1.1}
(Hardy-Littlewood-Sobolev inequality) Let $t,r>1$, and $0<\mu< N$ with $1/t+ 1/r+ \mu/N=2 ,~f\in L^t(\R^N)$  and $h\in L^r(\R^N)$. There exists a sharp constant  $C(t,r,N,\mu)>0$, independent of $f$ and $h$, such that
$$\int_{\R^N}\int_{\R^N}\frac{f(x)h(y)}{|x-y|^{\mu}}dxdy \leq C(t,r,N,\mu)||f||_t||h||_r.$$
\end{Pro}
If $f(x)=h(x)=u^p(x)$ in the HLS inequality,  the following integral
$$\int_{\R^N}\int_{\R^N}\frac{|u^p(x)||u^p(y)|}{|x-y|^{\mu}}dxdy $$
is well-defined provided $\frac{2N-\mu}{N}\leq p \leq \frac{2N-\mu}{N-2}$.  $2_{\mu*}=\frac{2N-\mu}{N}$ is called the lower critical Sobolev exponent and $2_{\mu}^*=\frac{2N-\mu}{N-2}$  the upper critical Sobolev exponent in the sense of the Hardy-Littlewood-Sobolev inequality. By the Sobolev inequality,
\begin{equation}\label{eq1.6}
\int_{\R^N}\int_{\R^N}\frac{|u(x)|^{2_{\mu}^*}|u(y)|^{2_{\mu}^*}}{|x-y|^{\mu}}dxdy \leq C(N,\mu)|\nabla u|_2^{2\cdot{2_{\mu}^*}},
\end{equation}
it is easy to see that the best constant  of \eqref{eq1.6} is related to the following nonlocal  equation
\begin{equation}\label{eq1.7}
-\Delta u = \Big(\int_{\R^N}\frac{|u(y)|^{2_{\mu}^*}}{|x-y|^{\mu}}dy\Big) |u|^{2_{\mu}^*-2}u,~~x\in  \mathbb{R}^{N},
\end{equation}
which is a special case of the  Hartree equation \eqref{eq1.2}.

In this  paper, we study the symmetry and monotonicity properties of the positive solutions of equation \eqref{eq1.2}. We need the following assumption of $f$:
\begin{itemize}
\item[$(f_1)$]  $f$ is $C^1$ and  convex with $f(0)=0$ as well as
$$f^\prime (s)\leq C_fs^{2_{\mu}^{*}-2},~ s>0, $$
for some  $C_f>0$, where $2_{\mu}^{*}=\frac{2N-\mu}{N-2}$ is the upper critical Sobolev exponent.
\end{itemize}
 \par
Our main result is as follows.
\begin{Thm}\label{Thm1.1}
Let $N\geq 3$, $N-2<\mu<N$  and
let $u\in \mathcal{D}^{1,2}(\R^N\setminus \Gamma)\cap C(\R^{N}\setminus \Gamma)$ be a solution to $\eqref{eq1.3}$ with a non-removable singularity on the compact set $\Gamma\subset \{x_1=0\}$.
Assume that $f$ satisfies $(f_1)$ and $\Gamma\subset \{x_1=0\}$ is of zero~$2$-capacity. Then, $u$ is symmetric with respect to the hyperplane $\{x\in \R^N: x_1=0\}$ and increasing in the $x_1$- direction in $\{x\in \R^N: x_1 <0\}$, i.e.,
$$u_{x_1}>0 ~~ in ~~ \{x_1 <0\}.$$
Moreover, the solution $u$ is radial and radially decreasing provided that the origin is the unique non-removable singularity of $u$.
\end{Thm}

In the second part of the paper we complement the above result by proving the  existence of a  positive singular solution to the following model problem
\begin{equation}\label{eq1.8}
-\Delta u= \Big(\int_{\R^N\setminus \{0\}} \frac{|u(y)|^p}{|x-y|^\mu}dy \Big) |u(x)|^q, \quad x\in \R^N\setminus \{0\}
\end{equation}
where $0<\mu<N$ and $p,q\geq1$.

Our second result is as follows.
\begin{Thm}\label{Thm1.2}
There exists a radially symmetric, positive singular solution $u(x)=\frac{A}{|x|^s}$ of \eqref{eq1.8},
where
$$s=\frac{N-\mu+2}{p-q+1}>0~~\text{and}~~A^{p+q-1}= \frac{s(N-2-s)\gamma(N-2+s(q-1))} {\gamma(N-\mu)\gamma(N-sp)}>0$$
with $0<sp<N$ and $2-N<s(q-1)<2$.
\end{Thm}

Observe that Theorem \ref{Thm1.2} provides an example of a singular solution to \eqref{eq1.3} as described in Theorem \ref{Thm1.1}.

\medskip

The paper is organized as follows. In section \ref{sec:prelim} we introduce some basic facts, in section \ref{sec:proof1} we prove symmetry and monotonicity of the solutions and in section \ref{sec:proof2} we study a model problem. A comparison principle is carried out in the appendix.

\par
\section{Notations and Preliminary results} \label{sec:prelim}
\par
We start by collecting some preliminary results. By the Hardy-Littlewood-Sobolev inequality we have the following result:
\begin{Pro}
Suppose that $f\in L^t(\R^N\setminus \Gamma)$ and $g\in L^r(\R^N\setminus \Gamma).$ Then we have
$$\aligned
\int_{\R^N\setminus\Gamma}\int_{\R^N\setminus\Gamma} \frac{f(x)g(y)}{|x-y|^{\mu}}dxdy \leq C(t,r,\mu,N)
\bigg( \int_{\R^N\setminus\Gamma} |f(x)|^{t}dx\bigg) ^{\frac{1}{t}}\bigg( \int_{\R^N\setminus\Gamma} |g(y)|^{r}dy\bigg) ^{\frac{1}{r}},
\endaligned$$
where $0<\mu<N,$ $\frac{1}{t}+\frac{1}{r}+\frac{\mu}{N}=2$ and $1-\frac{1}{t}-\frac{\mu}{N}<0<1-\frac{1}{t}.$
\end{Pro}
\begin{proof}
Note that $f\in L^t(\R^N\setminus \Gamma)$ if and only if $(1-\chi_{\Gamma})f\in L^t(\R^N)$, where $\chi_{\Gamma}$ denotes by the characteristic function of singular set $\Gamma$. It follows from the Hardy-Littlewood-Sobolev inequality that
$$\aligned
\int_{\R^N\setminus\Gamma}\int_{\R^N\setminus\Gamma} \frac{f(x)g(y)}{|x-y|^{\mu}}dxdy
&=\int_{\R^N}\int_{\R^N} \frac{f(x)(1-\chi_{\Gamma}(x))g(y)(1-\chi_{\Gamma}(y))}{|x-y|^{\mu}}dxdy\\
&\leq C(t,r,\mu,N)\bigg( \int_{\R^N} |f(x)(1-\chi_{\Gamma}(x))|^{t}dx\bigg) ^{\frac{1}{t}}\bigg( \int_{\R^N} |g(y)(1-\chi_{\Gamma})(y)|^{r}dy\bigg) ^{\frac{1}{r}}\\
&\leq C(t,r,\mu,N)\bigg( \int_{\R^N\setminus\Gamma} |f(x)|^{t}dx\bigg) ^{\frac{1}{t}}\bigg( \int_{\R^N\setminus\Gamma} |g(y)|^{r}dy\bigg)^{\frac{1}{r}},
\endaligned$$
where $0<\mu<N$ and $C(\mu,N,t,r)>0$ depends only by $t,r, N$ and $\mu$ as well as satisfying $\frac{1}{t}+\frac{1}{r}+\frac{\mu}{N}=2$ and $1-\frac{1}{t}-\frac{\mu}{N}<0<1-\frac{1}{t}$.
\end{proof}

\par
Frow now on, we will set, for $\lambda \in \mathbb{R} $,
$$\Sigma_{\lambda}=\{x\in \mathbb{R}^N: x_1<\lambda\}$$
and
$$x_{\lambda}=(2\lambda-x_1, x_2,\cdot \cdot \cdot, x_N)$$
which is the reflection of a point $x\in \mathbb{R}^N$ with respect to the hyperplane $T_{\lambda}:=\{x_1=\lambda\}$ and we denote by
$R_{\lambda}(x)=x_{\lambda}$ the reflection operator.
Also, we write
$$w_{\lambda}=u(x)-u_{\lambda}(x)=u(x)-u(x_{\lambda}).$$

\begin{Def} \label{Def2.1}
We say that  $u\in \mathcal{D}^{1,2}(\mathbb{R}^N\setminus \Gamma)\cap C(\mathbb{R}\setminus \Gamma)$ is a positive solution of $($\ref{eq1.1}$)$ if $u$ satisfies
\begin{equation}\label{eq2.1}
\begin{aligned}
\int_{\mathbb{R}^{N}\setminus\Gamma}\nabla u\cdot \nabla  \varphi dx = \int_{\R^N\setminus\Gamma}v(x)f(u(x))\varphi dx, \,\, \forall \varphi \in C_c^\infty(\mathbb{R}^N\setminus \Gamma),
\end{aligned}
\end{equation}
where $u>0$ in $\mathbb{R}^{N}\setminus\Gamma$ and $v(x)=\int_{\R^N\setminus\Gamma}\frac{F(u(y))}{|x-y|^{\mu}}dy.$
\end{Def}
Furthermore, for any $\varphi\in C_c^{\infty}(\R^N\setminus R_{\lambda}(\Gamma))$,
$$\aligned
\int_{\R^N\setminus R_{\lambda}( \Gamma)} \nabla u_{\lambda}(x)\nabla \varphi(x)dx
&=\int_{\R^N\setminus R_{\lambda}( \Gamma)} \nabla u_{\lambda}(x_{\lambda})\nabla\varphi(x_{\lambda})dx_{\lambda} =\int_{\R^N\setminus\Gamma} \nabla u(x)\nabla\varphi(x_{\lambda})dx\\
&=\int_{\R^N\setminus \Gamma} v(x)f(u(x))\varphi(x_{\lambda})dx
   =\int_{\R^N\setminus R_{\lambda}(\Gamma)} v(x)f(u(x))\varphi(x_{\lambda})dx_{\lambda} \\
& =\int_{\R^N\setminus R_{\lambda}(\Gamma)} v_{\lambda}(x)f(u_{\lambda}(x))\varphi(x)dx,
\endaligned$$
where
$$\aligned
v_{\lambda}(x)=v(x_{\lambda})&=\int_{\R^N\setminus \Gamma}\frac{F(u(y))}{|x_{\lambda}-y|^{\mu}}dy=\int_{\R^N\setminus R_{\lambda}(\Gamma)}\frac{F(u(y_\lambda))}{|x_{\lambda}-y_{\lambda}|^{\mu}}dy =\int_{\R^N\setminus R_{\lambda}(\Gamma)} \frac{F(u_{\lambda}(y))}{|x-y|^{\mu}}dy.
\endaligned$$
So we have
 \begin{equation}\label{eq2.2}
 \int_{\mathbb{R}^N\setminus R_{\lambda}(\Gamma)}\nabla u_{\lambda}\cdot \nabla \varphi dx=\int_{\mathbb{R}^N\setminus R_{\lambda}(\Gamma)}v_{\lambda}(x)f(u_{\lambda}(x))\varphi dx, \quad \forall \varphi\in C_c^{\infty}(\mathbb{R}^N\setminus R_{\lambda}(\Gamma)).
 \end{equation}

Assume that $K\subseteq\R^N$ is a compact set.
We recalled that the $r-$capacity of $K$ is defined as
$$   \mathop{Cap_r}\limits_{\mathbb{R}^N} (K)=\inf\Big\{\int_{\mathbb{R}^N}|\nabla \varphi|^r dx<+\infty: \varphi \in C_{c}^{\infty}(\mathbb{R}^N)\,\,  and \, \,\varphi\geq1 \, on\, K)\Big\} .$$
It follows from $\mathop{Cap_2}\limits_{\mathbb{R}^N}(\Gamma)=0$, the properties of the $r-$capacity of $K$ and the fact $\Gamma$ is compact set that
$$\mathop{Cap_2}\limits_{\mathbb{R}^N}( R_{\lambda}(\Gamma))=0.$$
Also notice that $\mathop{Cap_2}\limits_{\mathcal{B}_{ \varepsilon}^{\lambda}}(R_{\lambda}(\Gamma))=0$
for any open neighborhood $\mathcal{B}_{\varepsilon}^{\lambda}$ of $R_{\lambda}(\Gamma)$, where
$$\mathop{Cap_2}\limits_{\mathcal{B}^{\lambda}_{\varepsilon}}(R_{\lambda}(\Gamma))
      =\inf\Big\{\int_{\mathcal{B}_{\varepsilon}^{\lambda}}|\nabla \varphi|^2 dx<+\infty:\varphi \in C_{c}^{\infty}(\mathcal{B}_{\varepsilon}^{\lambda})\,\,  and \, \,  \varphi\geq1 \, on\, R_{\lambda}(\Gamma)\Big\}.$$
\par
Based on the properties  of zero capacity sets, we can construct a  cut-off function $\varphi_{\varepsilon}\in C^{0,1}(\mathbb{R}^N,[0,1])$ satisfies that $\varphi_{\varepsilon}=0$ in $\mathcal{B}_{\varepsilon}^{\lambda}$, $\varphi_{\varepsilon}=1$ outside $\mathcal{B}_{2\varepsilon}^{\lambda}$ and
$$\int_{\R^N}|\nabla \varphi_{\varepsilon}|^2 = \int_{\mathcal{B}_{2\varepsilon}^{\lambda}}|\nabla \varphi_{\varepsilon}|^2<\varepsilon.$$
One can refer to \cite{Biagi-Esposito-Vecchi2021JDE, Esposito-Farina-Sciunzi2018JDE}and the references therein  for more explicit details about the cut-off functions.
\par


\section{Symmetry and monotonicity} \label{sec:proof1}
The following result is the starting point of the proof of Theorem \ref{Thm1.1}.
\begin{Lem}\label{Lem3.1}
Under the assumptions of Theorem \ref{Thm1.1}, for $\lambda <0$, we have that
\begin{equation}\label{eq3.1}
\int_{\Sigma_{\lambda}}|\nabla w_{\lambda}^{+}|^2dx
\leq  C\big(\mu, N, ||u||_{L^{\infty}(\Gamma^{\lambda})},\,   ||u||_{L^{2^*}(\R^N\setminus \Gamma)}\big),
\end{equation}
where $w_{\lambda}^{+}=\max\{w_{\lambda},0\}$, $\Gamma^{\lambda}=\{x:dist(x,R_{\lambda}(\Gamma))<\frac{  |  \lambda|}{2}   \}$ and
 $w_{\lambda}^{+}\chi_{\Sigma_{\lambda}}\in \mathcal{D}^{1,2}(\mathbb{R}^N\setminus\Gamma)$ for any $\lambda <0$.
\end{Lem}
\begin{proof}
Let $\varphi_{\varepsilon}=\varphi_{\varepsilon}^{\lambda}$ be defined as before. In addition, let $\varphi_R\in C_c^{0,1}(\R^N)$ be a standard cut-off function such that  $\varphi_R=1$ in $B_R$ and $\varphi_R=0$ outside $B_{2R}$ with $|\nabla \varphi_R|\leq 2/R$, where the radius $R$ is large enough.  Define
$$\varphi :=w_{\lambda}^{+}\varphi_{\varepsilon}^2 { \varphi_R}^2   \chi_{\Sigma_{\lambda}\setminus R_{\lambda}(\Gamma)}=w_{\lambda}^{+}\varphi_{\varepsilon}^2 { \varphi_R}^2   \chi_{\Sigma_{\lambda}}.$$
Taking $\varphi$ as a test function in \eqref{eq2.1} and \eqref{eq2.2}, in the distributional sense it holds that
$$\aligned
\int_{\R^N\setminus\Gamma} \nabla u(x)\nabla \big[w^+_{\lambda}\varphi^2_{\varepsilon}\varphi^2_{R}\chi_{\Sigma_{\lambda}}\big]dx
=\int_{\R^N\setminus \Gamma} v(x)f(u(x))w^+_{\lambda}\varphi^2_{\varepsilon}\varphi^2_R dx.
\endaligned$$
and
$$\aligned
\nabla \big[w^+_{\lambda}\varphi^2_{\varepsilon}\varphi^2_{R}\chi_{\Sigma_{\lambda}}\big]
=\nabla w^+_{\lambda}\varphi^2_{\varepsilon}\varphi^2_{R}\chi_{\Sigma_{\lambda}}+ 2\nabla \varphi_{\varepsilon}w^+_{\lambda}\varphi_{\varepsilon}\varphi^2_{R}\chi_{\Sigma_{\lambda}}
+ 2\nabla \varphi_{R}w^+_{\lambda}\varphi^2_{\varepsilon}\varphi_{R}\chi_{\Sigma_{\lambda}}.
\endaligned$$
Thus it follows easily that
$$\aligned
&\int_{\R^N\setminus\Gamma}\nabla u(x)\nabla w^+_{\lambda}\varphi^2_{\varepsilon}\varphi^2_{R}\chi_{\Sigma_{\lambda}}dx=
\int_{\Sigma_{\lambda}}\nabla u(x)\nabla w^+_{\lambda}\varphi^2_{\varepsilon}\varphi^2_{R}dx =\int_{\Sigma_{\lambda}\setminus R_{\lambda}(\Gamma)}\nabla u(x)\nabla w^+_{\lambda}\varphi^2_{\varepsilon}\varphi^2_{R}dx \\
&=-2\int_{\Sigma_{\lambda}}\nabla u(x)\nabla \varphi_{\varepsilon}(x)w^+_{\lambda}\varphi_{\varepsilon}\varphi^2_Rdx
     -2\int_{\Sigma_{\lambda}}\nabla u(x)\nabla \varphi_{R}(x)w^+_{\lambda}\varphi^2_{\varepsilon}\varphi_Rdx \\
     &+\int_{\Sigma_{\lambda}}v(x)f(u(x))w^+_{\lambda}(x)\varphi^2_{\varepsilon}\varphi^2_Rdx \\
&=-2\int_{\Sigma_{\lambda}\setminus R_{\lambda}(\Gamma)}\nabla u(x)\nabla \varphi_{\varepsilon}(x)w^+_{\lambda}\varphi_{\varepsilon}\varphi^2_Rdx
     -2\int_{\Sigma_{\lambda}\setminus R_{\lambda}(\Gamma)}\nabla u(x)\nabla \varphi_{R}(x)w^+_{\lambda}\varphi^2_{\varepsilon}\varphi_Rdx \\
     &+\int_{\Sigma_{\lambda}\setminus R_{\lambda}(\Gamma)}v(x)f(u(x))w^+_{\lambda}(x)\varphi^2_{\varepsilon}\varphi^2_Rdx.
\endaligned$$
Moreover, we have
$$\aligned
\int_{\R^N\setminus R_{\lambda}(\Gamma)} \nabla u_{\lambda}(x)\nabla \big[w^+_{\lambda}\varphi^2_{\varepsilon}\varphi^2_{R}\chi_{\Sigma_{\lambda}}\big]dx
=\int_{\R^N\setminus R_{\lambda}(\Gamma)} v_{\lambda}(x)f(u_{\lambda}(x))w^+_{\lambda}\varphi^2_{\varepsilon}\varphi^2_R dx,
\endaligned$$
and
$$\aligned
&\int_{\R^N\setminus R_{\lambda}(\Gamma)}\nabla u_{\lambda}(x)\nabla w^+_{\lambda}\varphi^2_{\varepsilon}\varphi^2_{R}\chi_{\Sigma_{\lambda}}dx
=\int_{\Sigma_{\lambda}\setminus R_{\lambda}(\Gamma)}\nabla u_{\lambda}(x)\nabla w^+_{\lambda}\varphi^2_{\varepsilon}\varphi^2_{R}dx \\
&=-2\int_{\Sigma_{\lambda}\setminus R_{\lambda}(\Gamma)}\nabla u_{\lambda}(x)\nabla \varphi_{\varepsilon}(x)w^+_{\lambda}\varphi_{\varepsilon}\varphi^2_Rdx -2\int_{\Sigma_{\lambda}\setminus R_{\lambda}(\Gamma)}\nabla u_{\lambda}(x)\nabla \varphi_{R}(x)w^+_{\lambda}\varphi^2_{\varepsilon}\varphi_Rdx \\
&~~~~ +\int_{\Sigma_{\lambda}\setminus R_{\lambda}(\Gamma)}v_{\lambda}(x)f(u_{\lambda}(x))w^+_{\lambda}(x) \varphi^2_{\varepsilon}\varphi^2_Rdx.
\endaligned$$
By subtracting the two equalities we obtain that
$$\aligned
\int_{\Sigma_{\lambda}\setminus R_{\lambda}( \Gamma)} \nabla w_{\lambda}\nabla(w_{\lambda}^+) \varphi^2_{\varepsilon}\varphi^2_R dx
=& -2\int_{\Sigma_{\lambda}\setminus R_{\lambda}( \Gamma)}\nabla w_{\lambda}\nabla \varphi_{\varepsilon} (w_{\lambda}^+)\varphi_{\varepsilon}\varphi^2_{R}dx  \\
   &-2\int_{\Sigma_{\lambda}\setminus R_{\lambda}( \Gamma)}\nabla w_{\lambda}\nabla \varphi_{R} (w_{\lambda}^+)\varphi^2_{\varepsilon}\varphi_{R}dx  \\
   &+\int_{\Sigma_{\lambda}\setminus R_{\lambda}( \Gamma)}\Big(v(x)f(u(x))-v_{\lambda}(x)f(u_{\lambda}(x))\Big)
  (w_{\lambda}^+) \varphi^2_R\varphi^2_{\varepsilon} dx.
\endaligned$$
That is,
\begin{equation}\label{eq3.2}
\begin{aligned}
\int_{\Sigma_{\lambda}\setminus R_{\lambda}(\Gamma)}|\nabla w_{\lambda}^{+}|^2\varphi_{\varepsilon}^2 \varphi_R^2 dx
=&-2\int_{\Sigma_{\lambda}\setminus R_{\lambda}(\Gamma)}\nabla w_{\lambda}^{+}\nabla \varphi_{\varepsilon}
(w_{\lambda}^+)\varphi_R^2 \varphi_\varepsilon dx \\
 &-2\int_{\Sigma_{\lambda}}\nabla w_{\lambda}^{+}\nabla \varphi_{R} (w_{\lambda}^+) \varphi_{\varepsilon}^2 \varphi_R dx    \\
 &+\int_{\Sigma_{\lambda}\setminus R_{\lambda}(\Gamma)}\big[v(x)-v_{\lambda}(x)\big]f(u(x))
(w_{\lambda}^+)\varphi^2_{\varepsilon}\varphi^2_{R}dx  \\
& +\int_{\Sigma_{\lambda}\setminus R_{\lambda}(\Gamma)}v_{\lambda}(x)\big[f(u(x))-f(u_{\lambda}(x))\big]
(w_{\lambda}^+)\varphi^2_{\varepsilon}\varphi^2_{R}dx   \\
 =&:I_1 +I_2  +I_3+I_4.
\end{aligned}
\end{equation}
The estimates about terms $I_1$ and $I_2$ are very similar to the proof of  \cite[Theorem 1.3]{Sciunzi2017JMPA}, here we directly use the following results,
\begin{equation}\label{eq3.3}
\begin{aligned}
|I_1|
\leq
\frac{1}{4}\int_{\Sigma_{\lambda}}|\nabla w_{\lambda}^{+}|^2\varphi_{\varepsilon}^2  \varphi_R^2 dx +
 C\Big(N,\, ||u||_{L^{\infty}(\Gamma^{\lambda})}\Big),  \\
\end{aligned}
\end{equation}
and
\begin{equation}\label{eq3.4}
\begin{aligned}
|I_2|\leq
\frac{1}{4}\int_{\Sigma_{\lambda}}|\nabla w_{\lambda}^{+}|^2\varphi_{\varepsilon}^2  \varphi_R^2 dx +
+C\Big(N, \, ||u||_{L^{2^*}(\Sigma_{\lambda})}  \Big).    \\
\end{aligned}
\end{equation}
where $2\varepsilon< |\lambda|/2$ such that $\Sigma_{\lambda}\cap (\mathcal{B}_{2\varepsilon}^{\lambda}\setminus\mathcal{B}_{\varepsilon}^{\lambda} )\subset \mathcal{B}_{2\varepsilon}^{\lambda}\subset\Gamma^{\lambda}$.
Next we estimate the term $I_3$, it follows that
\begin{align*}
v(&x)-v_{\lambda}(x)
=\int_{\R^N\setminus \Gamma} \frac{F(u(y))}{|x-y|^{\mu}}dy-\int_{\R^N\setminus R_{\lambda}(\Gamma)} \frac{F(u_{\lambda}(y))}{|x-y|^{\mu}}dy \\
=&\int_{\R^N\setminus \Gamma} \frac{F(u(y))}{|x-y|^{\mu}}dy-\int_{\R^N\setminus\Gamma}
\frac{F(u_{\lambda}(y_{\lambda}))}{|x-y_{\lambda}|^{\mu}}dy \\
=&\int_{\R^N\setminus\Gamma} F(u(y)) \bigg(\frac{1}{|x-y|^\mu}-\frac{1}{|x_{\lambda}-y|^{\mu}}\bigg)dy\\
=&\int_{\Sigma_{\lambda}} F(u(y)) \bigg(\frac{1}{|x-y|^\mu}-\frac{1}{|x_{\lambda}-y|^{\mu}}\bigg)dy +\int_{\Sigma_{\lambda}^c\setminus\Gamma} F(u(y)) \bigg(\frac{1}{|x-y|^\mu}-\frac{1}{|x_{\lambda}-y|^{\mu}}\bigg)dy\\
=&\int_{\Sigma_{\lambda}} F(u(y)) \bigg(\frac{1}{|x-y|^\mu}-\frac{1}{|x_{\lambda}-y|^{\mu}}\bigg)dy +\int_{\Sigma_{\lambda}\setminus R_{\lambda}(\Gamma)} F(u(y_{\lambda})) \bigg(\frac{1}{|x-y_{\lambda}|^\mu}-\frac{1}{|x_{\lambda}-y_{\lambda}|^{\mu}}\bigg)dy\\
=&\int_{\Sigma_{\lambda}\setminus R_{\lambda}(\Gamma)} F(u(y)) \bigg(\frac{1}{|x-y|^\mu}-\frac{1}{|x_{\lambda}-y|^{\mu}}\bigg)dy +\int_{\Sigma_{\lambda}\setminus R_{\lambda}(\Gamma)} F(u_{\lambda}(y)) \bigg(\frac{1}{|x_{\lambda}-y|^\mu}-\frac{1}{|x-y|^{\mu}}\bigg)dy\\
=&\int_{\Sigma_{\lambda}\setminus R_{\lambda}(\Gamma)}\Big(F(u(y))-F(u_{\lambda}(y))\Big)\bigg(\frac{1}{|x-y|^{\mu}}-\frac{1}{|x_{\lambda}-y|^{\mu}}\bigg)dy.
\end{align*}
On the other hand, via the mean value theorem, it yields that
\begin{equation}\label{eq3.5}
F(u)-F(u_\lambda)=f(\xi_{\lambda}(y))(u(y)-u_{\lambda}(y))\leq f(u(y))(u(y)-u_{\lambda}(y))
\end{equation}
and
\begin{equation}\label{eq3.6}
f(u)-f(u_{\lambda})=f'(\zeta_{\lambda}(x))(u(x)-u_{\lambda}(x))\leq f'(u)(u(x)-u_{\lambda}(x)),
\end{equation}
for $u_{\lambda}(y)<\xi_{\lambda}(y)<u(y)$ and $u_{\lambda}(x)<\zeta_{\lambda}(x)<u(x)$ respectively.
Combining with the monotonicity of $f$ and $F$, then it follows from H\"{o}lder's and Hardy-Littlewood-Sobolev  inequality that
\begin{equation}\label{eq3.7}
\begin{aligned}
I_3
=&\int_{\Sigma_{\lambda}\setminus R_{\lambda}(\Gamma)} \big(v(x)-v_{\lambda}(x)\big) f(u(x))(w_{\lambda})^+(x)\varphi^2_{\varepsilon}\varphi^2_{R}dx\\
=&\int_{\Sigma_{\lambda}\setminus R_{\lambda}(\Gamma)}\bigg(\int_{\Sigma_{\lambda}\setminus R_{\lambda}(\Gamma)}\Big(F(u(y))-F(u_{\lambda}(y))\Big)\bigg(\frac{1}{|x-y|^{\mu}}
-\frac{1}{|x_{\lambda}-y|^{\mu}}\bigg)dy\\
&\times f(u(x))(w_{\lambda})^+(x)\varphi^2_{\varepsilon}(x)\varphi^2_R(x)\bigg)dx \\
\leq & \int_{\Sigma_{\lambda}\setminus R_{\lambda}(\Gamma)}\bigg( \int_{\Sigma^u_{\lambda}\setminus R_{\lambda}(\Gamma)}
\Big(F(u(y))-F(u_{\lambda}(y))\Big) \bigg(\frac{1}{|x-y|^{\mu}}-\frac{1}{|x_{\lambda}-y|^{\mu}}\bigg)dy \\
     &\times f(u(x))(w_{\lambda})^+(x)\varphi^2_{\varepsilon}\varphi^2_R \bigg)dx \\
\leq&  \int_{\Sigma_{\lambda}\setminus R_{\lambda}(\Gamma)}\bigg( \int_{\Sigma^u_{\lambda}} \frac{C(f)u^{2^*_{\mu}-1}(y)(w^+_{\lambda})(y)}{|x-y|^{\mu}}dy C_f u^{2^*_{\mu}-1}(x)(w_{\lambda})^+(x)\varphi^2_{\varepsilon}\varphi^2_R \bigg)dx \\
\leq& C(f,C_f) \iint\limits_{\Sigma_{\lambda}\times \Sigma_{\lambda}} \frac{u^{2^*_{\mu}}(y)u^{2^*_{\mu}-1}(x)}{|x-y|^{\mu}}(w_{\lambda})^+(x) \varphi^2_{\varepsilon}\varphi^2_R dxdy \\
\leq& C(f,C_f) C(\mu, N) \Big(\int_{\Sigma_{\lambda}}u^{2^*}(y)dy\Big)^{\frac{2^*_{\mu}}{2^*}}
\bigg(\int_{\Sigma_{\lambda}}\Big(u^{2^*_{\mu}-1}(x)(w_{\lambda})^+(x)
\varphi^2_{\varepsilon}(x)\varphi^2_{R}(x)\Big)^{\frac{2^*}{2^*_{\mu}}}
dx\bigg)^{\frac{2^*_{\mu}}{2^*}}\\
\end{aligned}
\end{equation}

\begin{equation*}
\begin{aligned}
\leq& C(f) C(\mu, N) ||u||^{2^*_{\mu}}_{L^{2^*}(\Sigma_{\lambda})} \bigg(\int_{\Sigma_{\lambda}\cap B_{2R}}
\Big(u^{2^*_{\mu}-1}(x)(w_{\lambda})^+(x)\Big)^{\frac{2^*}{2^*_{\mu}}}dx\bigg)^{\frac{2^*_{\mu}}{2^*}} \\
\leq& C(\mu, N, f) ||u||^{2^*_{\mu}}_{L^{2^*}(\Sigma_{\lambda})}\Big(\int_{\Sigma_{\lambda}\cap B_{2R}}u^{2^*}(x)dx\Big)^{\frac{2^*_{\mu}-1}{2^*}}
      \times \Big(\int_{\Sigma_{\lambda}\cap B_{2R}} (w^+_{\lambda})^{2^*}(x)dx\Big)^{\frac{1}{2^*}}\\
\leq& C(\mu, N,f) ||u||^{2^*_{\mu}}_{L^{2^*}(\Sigma_{\lambda})}||u||^{2^*_{\mu}-1}_{L^{2^*}(\Sigma_{\lambda})}||w^+_{\lambda}||_{L^{2^*}(\Sigma_{\lambda})}\\
\leq& C(\mu, N,f) ||u||^{2^*_{\mu}}_{L^{2^*}(\Sigma_{\lambda})}||u||^{2^*_{\mu}-1}_{L^{2^*}(\Sigma_{\lambda})}||u||_{L^{2^*}(\Sigma_{\lambda})}\\
\leq& C(\mu, N,f) ||u||^{2\cdot2^*_{\mu}}_{L^{2^*}(\Sigma_{\lambda})},
\end{aligned}
\end{equation*}
where $\Sigma^u_{\lambda}:=\{x\in \Sigma_{\lambda}\big| \, u(x)>u_{\lambda}(x)\}$. In a similar way, we may also estimate the upper bound for $I_4$. We first observe that
\begin{flalign*}
v_{\lambda}(x)=&
 \int_{\Sigma_{\lambda}\setminus R_{\lambda}(\Gamma)} \frac{F(u_{\lambda}(y))}{|x-y|^{\mu}}dy+\int_{\Sigma^c_{\lambda}} \frac{F(u_{\lambda}(y))}{|x-y|^{\mu}}dy\\
=& \int_{\Sigma_{\lambda}\setminus R_{\lambda}(\Gamma)} \frac{F(u_{\lambda}(y))}{|x-y|^{\mu}}dy+\int_{\Sigma^c_{\lambda}\setminus\Gamma} \frac{F(u_{\lambda}(y_{\lambda}))}{|x-y_{\lambda}|^{\mu}}dy_{\lambda}\\
=& \int_{\Sigma_{\lambda}\setminus R_{\lambda}(\Gamma)} \frac{F(u_{\lambda}(y))}{|x-y|^{\mu}}dy+\int_{\Sigma_{\lambda}\setminus R_{\lambda}(\Gamma)} \frac{F(u_{\lambda}(y_{\lambda}))}{|x-y_{\lambda}|^{\mu}}dy\\
=& \int_{\Sigma_{\lambda}\setminus R_{\lambda}(\Gamma)} \frac{F(u_{\lambda}(y))}{|x-y|^{\mu}}dy+\int_{\Sigma_{\lambda}\setminus R_{\lambda}(\Gamma)} \frac{F(u(y))}{|y-x_{\lambda}|^{\mu}}dy\\
=& \int_{\Sigma_{\lambda}\setminus R_{\lambda}(\Gamma)}\Big(\frac{F(u_{\lambda}(y))}{|x-y|^{\mu}}+\frac{F(u(y))}{|x_{\lambda}-y|^{\mu}}\Big)dy\\
\leq& \int_{\Sigma_{\lambda}\setminus R_{\lambda}(\Gamma)}\frac{F(u_{\lambda}(y))+F(u(y))}{|x-y|^{\mu}}dy\\
\leq& C(f) \int_{\Sigma_{\lambda}\setminus R_{\lambda}(\Gamma)}\frac{u^{2^*_{\mu}}_{\lambda}(y)+u^{2^*_{\mu}}(y)}{|x-y|^{\mu}}dy,\\
\end{flalign*}
Therefore, we have
\begin{equation}\label{eq3.8}
\begin{aligned}
I_4=&\int_{\Sigma_{\lambda}\setminus R_{\lambda}(\Gamma)}v_{\lambda}(x)\big[f(u(x))-f(u_{\lambda}(x))\big](w_{\lambda})^+(x)\varphi^2_{\varepsilon}\varphi^2_{R}dx   \\
\leq& C(f) \int_{\Sigma_{\lambda}\setminus R_{\lambda}(\Gamma)}\int_{\Sigma_{\lambda}\setminus R_{\lambda}(\Gamma)}\frac{u^{2^*_{\mu}}_{\lambda}(y)+u^{2^*_{\mu}}(y)}{|x-y|^{\mu}}
u^{2^*_{\mu}-2}(x)(w^+_{\lambda})^2(x)\varphi^2_{\varepsilon}(x)\varphi^2_{R}(x)dxdy \\
\leq& C(f)C(\mu, N) \bigg(\Big(\int_{\Sigma_{\lambda}\setminus R_{\lambda}(\Gamma)}u^{2^*}_{\lambda}dy\Big)^{\frac{2^*_{\mu}}{2^*}}+
\Big(\int_{\Sigma_{\lambda}}|u|^{2^*}dy\Big)^{\frac{2^*_{\mu}}{2^*}}\bigg) \\
&\times \Big(\int_{\Sigma_{\lambda}\cap B_{2R}}\big|u^{2^*_{\mu}-2}(w^+_{\lambda})^2(x)\big|^{\frac{2^*}{2^*_{\mu}}}dx\Big)^{\frac{2^*_{\mu}}{2^*}}\\
\leq& 2C(f) C(\mu, N) ||u||^{2^*_{\mu}}_{L^{2^*}(\R^N\setminus\Gamma)} \Big(\int_{\Sigma_{\lambda}\cap B_{2R}}
u^{2^*}dx\Big)^{\frac{2^*_{\mu}}{2^*}}\\
\leq& C(\mu, N, f) ||u||^{2^*_{\mu}}_{L^{2^*}(\R^N\setminus\Gamma)}||u||^{2^*_{\mu}}_{L^{2^*}(\Sigma_{\lambda})},
\end{aligned}
\end{equation}
where $ C(\mu, N, f)>0$ depends  only on $N, f$ and $\mu$. Talking $($\ref{eq3.3}$)$-$($\ref{eq3.4}$)$ and $($\ref{eq3.7}$)$-$($\ref{eq3.8}$)$ into $($\ref{eq3.2}$)$, we deduce that
$$\aligned
\int_{\Sigma_{\lambda}\setminus R_{\lambda}(\Gamma)}|\nabla w_{\lambda}^{+}|^2\varphi_{\varepsilon}^2 \varphi_R^2 dx
\leq& \frac{1}{4}\int_{\Sigma_{\lambda}}|\nabla w_{\lambda}^{+}|^2\varphi_{\varepsilon}^2  \varphi_R^2 dx +
 C\Big(N,\, ||u||_{L^{\infty}(\Gamma^{\lambda})}\Big)\\
 &+\frac{1}{4}\int_{\Sigma_{\lambda}}|\nabla w_{\lambda}^{+}|^2\varphi_{\varepsilon}^2  \varphi_R^2 dx +
+C\Big(N, \, ||u||_{L^{2^*}(\Sigma_{\lambda})}  \Big)\\
&+C(\mu, N,f) ||u||^{2\cdot2^*_{\mu}}_{L^{2^*}(\Sigma_{\lambda})}+C(\mu, N, f) ||u||^{2^*_{\mu}}_{L^{2^*}(\R^N\setminus\Gamma)}||u||^{2^*_{\mu}}_{L^{2^*}(\Sigma_{\lambda})},
\endaligned$$
i.e.  $$\aligned
\int_{\Sigma_{\lambda}\setminus R_{\lambda}(\Gamma)}|\nabla w_{\lambda}^{+}|^2\varphi_{\varepsilon}^2 \varphi_R^2 dx
\leq& 2 C\Big(N,\, ||u||_{L^{\infty}(\Gamma^{\lambda})}\Big)+2C\Big(N, \, ||u||_{L^{2^*}(\Sigma_{\lambda})}  \Big)\\
&+2C(\mu, N,f) ||u||^{2\cdot2^*_{\mu}}_{L^{2^*}(\Sigma_{\lambda})}+2C(\mu, N, f) ||u||^{2^*_{\mu}}_{L^{2^*}(\R^N\setminus\Gamma)}||u||^{2^*_{\mu}}_{L^{2^*}(\Sigma_{\lambda})}.
\endaligned$$
Notice that $\varphi_{\varepsilon},\varphi_R\to 1,$ as $(R,\varepsilon)\to(+\infty,0)$,   it follows from the  Fatou's lemma that
$$\aligned
\int_{\Sigma_{\lambda}}|\nabla w_{\lambda}^{+}|^2dx
\leq& \liminf\limits_{(R,\varepsilon)\to(+\infty,0)}\int_{\Sigma_{\lambda}\setminus R_{\lambda}(\Gamma)}|\nabla w_{\lambda}^{+}|^2\varphi_{\varepsilon}^2 \varphi_R^2 dx\\
\leq& 2 C\Big(N,\, ||u||_{L^{\infty}(\Gamma^{\lambda})}\Big)+2C\Big(N, \, ||u||_{L^{2^*}(\Sigma_{\lambda})}  \Big)\\
&+2C(\mu, N,f) ||u||^{2\cdot2^*_{\mu}}_{L^{2^*}(\Sigma_{\lambda})}+2C(\mu, N, f) ||u||^{2^*_{\mu}}_{L^{2^*}(\R^N\setminus\Gamma)}||u||^{2^*_{\mu}}_{L^{2^*}(\Sigma_{\lambda})}\\
\leq&  C\big(\mu, N, ||u||_{L^{\infty}(\Gamma^{\lambda})},||u||_{L^{2^*}(\Sigma_{\lambda})},||u||_{L^{2^*}(\R^N\setminus \Gamma)}\big).
\endaligned$$
The desired estimates follow and the proof is completed.
\end{proof}

\par
Note that $0\leq (w_{\lambda})^+\leq u$, so we deduce from  Lemma \ref{Lem3.1}  that, for $\lambda <0$,
\begin{equation}\label{eq3.9}
\int_{\Sigma_{\lambda}}|\nabla w_{\lambda}^{+}|^2(w_{\lambda})^+dx<+\infty.
\end{equation}

In what follows, we are devoted to the proof of Theorem \ref{Thm1.1}.
\par
$\mathbf{Proof~of~Theorem~\ref{Thm1.1}.}$~~

\textbf{Step1:} we assert that, for $\lambda <0$ with $|\lambda|$ sufficiently large, it follows that $u\leq u_{\lambda}$ in $\Sigma_{\lambda}\setminus R_{\lambda}(\Gamma)$.  \\
By using the same notations and construction for $\varphi_{\varepsilon}$ and $\varphi_R$, and then consider the following test function $\psi:=(w^+_{\lambda})^2\varphi^2_{\varepsilon}\varphi^2_R\chi_{\Sigma_{\lambda}}$. Thus we  get that
$$\aligned
\nabla \psi(x)=&2\nabla (w_{\lambda})^+ (w_{\lambda})^+\varphi^2_{\varepsilon}\varphi^2_R\chi_{\Sigma_{\lambda}}
+ 2\nabla \varphi_{\varepsilon} (w_{\lambda}^+)^2\varphi_{\varepsilon}\varphi^2_R\chi_{\Sigma_{\lambda}}\\
&+ 2\nabla \varphi_{R} (w_{\lambda}^+)^2\varphi^2_{\varepsilon}\varphi_R\chi_{\Sigma_{\lambda}}.
\endaligned$$
Thus, together with
\begin{equation*}\label{eq3.10}
\begin{aligned}
\int_{\R^N\setminus \Gamma}\nabla u \nabla \psi dx=\int_{\R^N}v(x)f(u(x))\psi(x)dx
\end{aligned}
\end{equation*}
and
\begin{equation*}\label{eq3.11}
\begin{aligned}
\int_{\R^N\setminus R_{\lambda}(\Gamma)}\nabla u_{\lambda} \nabla \psi dx=\int_{\R^N\setminus R_{\lambda}(\Gamma)}f(u_{\lambda}(x))\psi(x)dx,
\end{aligned}
\end{equation*}
by repeating the proof of Lemma \ref{Lem3.1}, we have that
$$\aligned
\int_{\Sigma_{\lambda}\setminus R_{\lambda}(\Gamma)}\nabla u(x) \nabla\big[(w^+_{\lambda})^2\varphi^2_{\varepsilon}\varphi^2_{R}\big]dx
=\int_{\Sigma_{\lambda}\setminus R_{\lambda}(\Gamma)}v(x)f(u(x))[(w_{\lambda}^+)]^2(x)\varphi^2_{\varepsilon}\varphi^2_{R}\big]dx
\endaligned$$
and
$$\aligned
\int_{\Sigma_{\lambda}\setminus R_{\lambda}(\Gamma)}\nabla u_{\lambda}(x) \nabla\big[(w^+_{\lambda})^2\varphi^2_{\varepsilon}\varphi^2_{R}\big]dx
=\int_{\Sigma_{\lambda}\setminus R_{\lambda}(\Gamma)}v_{\lambda}(x)f(u_{\lambda}(x))[ (w_{\lambda}^+)]^2(x)\varphi^2_{\varepsilon}\varphi^2_{R}\big]dx.
\endaligned$$
Subtracting, we get
\begin{equation}\label{eq3.12}
\begin{aligned}
2\int_{\Sigma_{\lambda}}\big|\nabla w^+_{\lambda}\big|^2 w^+_{\lambda}\varphi^2_{\varepsilon}\varphi^2_R dx
=&-2\int_{\Sigma_{\lambda}}\nabla w^+_{\lambda}\nabla \varphi_{\varepsilon} \varphi_{\varepsilon}(w^+_{\lambda})^2\varphi^2_R dx
-2\int_{\Sigma_{\lambda}}\nabla w^+_{\lambda}\nabla \varphi_R \varphi_R(w^+_{\lambda})^2\varphi^2_{\varepsilon} dx \\
    &+\int_{\Sigma_{\lambda}\setminus R_{\lambda}(\Gamma)} \big[v(x)-v_{\lambda}(x)\big] f(u(x)) (w^+_{\lambda})^2(x)\varphi^2_{\varepsilon}\varphi^2_Rdx  \\
    &+\int_{\Sigma_{\lambda}\setminus R_{\lambda}(\Gamma)}v_{\lambda}(x) \big[f(u(x))-f(u_{\lambda}(x))\big](w^+_{\lambda})^2(x)\varphi^2_{\varepsilon}\varphi^2_Rdx  \\
=&:\tilde{ I_1}+\tilde{I_2}+\tilde{I_3}+\tilde{I_4}.
\end{aligned}
\end{equation}
Fix $R>0$ sufficiently large and $\varepsilon$ small enough, we estimate $\tilde{I_1}$, $\tilde{I_2}$ and $\tilde{I_3}$ term by term. For $\tilde{I_1}$, by  Cauchy's inequality we get
\begin{equation}\label{eq3.13}
\begin{aligned}
|\tilde{I_1}|=& 2\Big|\int_{\Sigma_{\lambda}} \nabla w^+_{\lambda}\nabla\varphi_{\varepsilon} \varphi_{\varepsilon}(w^+_{\lambda})^2\varphi^2_R dx\Big|
\leq 2\int_{\Sigma_{\lambda}}\big|\nabla w^+_{\lambda}\big| |\nabla \varphi_{\varepsilon}|(w^+_{\lambda})^2\varphi_{\varepsilon}\varphi^2_R dx\\
\leq & 2\int_{\Sigma_{\lambda}\cap(\mathcal{B}^{\lambda}_{2\varepsilon}\setminus\mathcal{B}^{\lambda}_{\varepsilon})\cap(B_{2R}\setminus B_R)}  \big|\nabla w^+_{\lambda}\big| |\nabla \varphi_{\varepsilon}|(u)^2dx \\
\leq& 2||u||^2_{L^{\infty}(\mathcal{B}^{\lambda}_{2\varepsilon}\setminus\mathcal{B}^{\lambda}_{\varepsilon})} \int_{\Sigma_{\lambda}\cap(\mathcal{B}^{\lambda}_{2\varepsilon}\setminus\mathcal{B}^{\lambda}_{\varepsilon})}
   \big|\nabla (w_{\lambda}^+)\big| |\nabla \varphi_{\varepsilon}| dx\\
\leq&  2||u||^2_{L^{\infty}(\Gamma^{\lambda})}\Big(\int_{\Sigma_{\lambda}\cap
(\mathcal{B}^{\lambda}_{2\varepsilon}\setminus\mathcal{B}^{\lambda}_{\varepsilon})}
\big|\nabla w^+_{\lambda}\big|^2dx\Big)^{\frac{1}{2}}\Big(\int_{\mathcal{B}^{\lambda}_{2\varepsilon}}|\nabla\varphi_{\varepsilon}|^2dx\Big)^{\frac{1}{2}}\\
\leq & 2\sqrt{\varepsilon}||u||^2_{L^{\infty}(\Gamma^{\lambda})}
\int_{\Sigma_{\lambda}}\Big(|\nabla w^+_{\lambda}|^2dx\Big)^{\frac{1}{2}},
\end{aligned}
\end{equation}
where $\Gamma^{\lambda}:=\{x: dist(x, R_{\lambda}(\Gamma))<\frac{|\lambda|}{2}\}$ with $\varepsilon<\frac{|\lambda|}{4}.$  It follows from Lemma $\ref{Lem3.1}$ that  $|\tilde{I_1}|\to0$ as $\varepsilon\to0^+$.
By  the  H\"{o}lder's inequality, we get the following  estimate of $\tilde{I_2}$,
\begin{equation}\label{eq3.14}
\begin{aligned}
|\tilde{I_2}|
=& 2\Big|\int_{\Sigma_{\lambda}\setminus R_{\lambda}(\Gamma)}\nabla w^+_{\lambda}\nabla \varphi_R \varphi_R(w^+_{\lambda})^2\varphi^2_{\varepsilon}dx\Big|\\
\leq &2 \int_{\Sigma_{\lambda}}\big|\nabla w^+_{\lambda}\big| |\nabla \varphi_R| \varphi_R (w^+_{\lambda})^2\varphi^2_{\varepsilon} dx \\
\leq & 2\Big(\int_{\Sigma_{\lambda}}\big|\nabla w^+_{\lambda}\big|^2dx\Big)^{\frac{1}{2}}\times \Big(\int_{\Sigma_{\lambda}\cap(B_{2R}\setminus B_R)}|\nabla \varphi_{R}|^2(w^+_{\lambda})^4dx\Big)^{\frac{1}{2}}\\
\leq & 2 ||u||_{L^{\infty}(\Sigma_{\lambda})}\Big(\int_{\Sigma_{\lambda}}\big|\nabla w^+_{\lambda}\big|^2dx\Big)^{\frac{1}{2}}\times \Big(\int_{\Sigma_{\lambda}\cap(B_{2R}\setminus B_R)}|\nabla \varphi_{R}|^2(w^+_{\lambda})^2dx\Big)^{\frac{1}{2}}\\
\leq & 2 ||u||_{L^{\infty}(\Sigma_{\lambda})}\Big(\int_{\Sigma_{\lambda}}\big|\nabla w^+_{\lambda}\big|^2dx\Big)^{\frac{1}{2}}
\Big(\int_{B_{2R}\setminus B_R}|\nabla \varphi_R|^Ndx\Big)^{\frac{1}{N}} \Big(\int_{\Sigma_{\lambda}\cap(B_{2R}\setminus B_R)}|u|^{2^*}dx\Big)^{\frac{1}{2^*}} \\
\leq & 2 C(N)||u||_{L^{\infty}(\Sigma_{\lambda})}\Big(\int_{\Sigma_{\lambda}}\big|\nabla w^+_{\lambda}\big|^2dx\Big)^{\frac{1}{2}}
\Big(\int_{\Sigma_{\lambda}\cap(B_{2R}\setminus B_R)}|u|^{2^*}dx\Big)^{\frac{1}{2^*}} .
\end{aligned}
\end{equation}
Notice that the last integral in \eqref{eq3.14} above tends to zero when $R\to+\infty$ since $u\in L^{2^*}(\Sigma_{\lambda}) $ due to Lemma \ref{Lem3.1}. Hence, we have $|\tilde{I_2}|\to0$ as $R\to +\infty$. Similar to the proof Lemma \ref{Lem3.1}, combining the mean value theorem and $(f_1)$, we obtain
\begin{equation}\label{eq3.15}
\begin{aligned}
\tilde{I_3}=&\int_{\Sigma_{\lambda}\setminus R_{\lambda}(\Gamma)} \big(v(x)-v_{\lambda}(x)\big)f(u(x))(w^+_{\lambda})^2(x)\varphi^2_{\varepsilon}\varphi^2_Rdx\\
=& \int_{\Sigma_{\lambda}\setminus R_{\lambda}(\Gamma)} \bigg(\int_{\Sigma_{\lambda}\setminus R_{\lambda}(\Gamma)}\Big(F(u(y))-F(u_{\lambda}(y))\Big)\Big(\frac{1}{|x-y|^{\mu}}-\frac{1}{|x_{\lambda}-y|^{\mu}}\Big)dy\\
     &\times f(u(x))(w^+_{\lambda})^2(x)\varphi^2_{\varepsilon}\varphi^2_R\bigg)dx  \\
\leq& \int_{\Sigma_{\lambda}\setminus R_{\lambda}(\Gamma)} \bigg(\int_{\Sigma^u_{\lambda}\setminus R_{\lambda}(\Gamma)}
\Big(F(u(y))-F(u_{\lambda}(y))\Big) \Big(\frac{1}{|x-y|^{\mu}}-\frac{1}{|x_{\lambda}-y|^{\mu}}\Big)dy \\
&\times  f(u(x))(w^+_{\lambda})^2(x)\varphi^2_{\varepsilon}\varphi^2_R\bigg)dx  \\
\leq&  \int_{\Sigma_{\lambda}\setminus R_{\lambda}(\Gamma)} \bigg(\int_{\Sigma^u_{\lambda}\setminus R_{\lambda}(\Gamma)}
\Big(F(u(y))-F(u_{\lambda}(y))\Big) \Big(\frac{1}{|x-y|^{\mu}}\Big)dy \\
 &\times  f(u(x))(w^+_{\lambda})^2(x)\varphi^2_{\varepsilon}\varphi^2_R\bigg)dx  \\
\leq&  \int_{\Sigma_{\lambda}\setminus R_{\lambda}(\Gamma)} \bigg( \int_{\Sigma^u_{\lambda}\setminus R_{\lambda}(\Gamma)}\frac{C(f)u^{2^*_{\mu}-1}(y)(w^+_{\lambda})(y)}{|x-y|^{\mu}}dy C_f u^{2^*_{\mu}-1}(x)(w^+_{\lambda})^2(x)\varphi^2_{\varepsilon}\varphi^2_R \bigg)dx \\
\leq& C(f,C_f) \iint_{\Omega\times\Omega}\frac{u^{2^*_{\mu}-1}(y)( w^+_{\lambda})(y)u^{2^*_{\mu}-1}(x)(w^+_{\lambda})^2(x)}{|x-y|^{\mu}}dxdy,\\
\end{aligned}
\end{equation}
where $\Omega:=\Sigma_{\lambda}\setminus R_{\lambda}(\Gamma).$  Note that
$$\aligned
u^{2^*_{\mu}-1}(y)(w^+_{\lambda})(y)u^{2^*_{\mu}-1}(x)(w^+_{\lambda})^2(x)
\leq u^{2^*_{\mu}}(y)u^{2^*_{\mu}-2}(x)(w^+_{\lambda})^3(x)
\endaligned$$
for $u(x)(w^+_{\lambda})(y)\leq u(y)(w^+_{\lambda})(x)$  and
$$\aligned
u^{2^*_{\mu}-1}(y)(w^+_{\lambda})(y)u^{2^*_{\mu}-1}(x)(w^+_{\lambda})^2(x)
\leq u^{2^*_{\mu}+1}(x)u^{2^*_{\mu}-3}(y)(w^+_{\lambda})^3(y)
\endaligned$$
for $u(x)(w^+_{\lambda})(y)\geq u(y)(w^+_{\lambda})(x).$
Define $A:=\big\{(x,y)\in \Omega\times\Omega \big| u(x)(w^+_{\lambda})(y)\leq u(y)(w^+_{\lambda})(x)\big\}$ and $A^c:=(\Omega\times\Omega )\setminus A,$ then
$$\aligned
\tilde{I_3}
\leq & C(f,C_f) \int_{\Omega\times\Omega}\frac{u^{2^*_{\mu}-1}(y)(w^+_{\lambda})(y)u^{2^*_{\mu}-1}(x)(w^+_{\lambda})^2(x)}{|x-y|^{\mu}}dxdy\\
\leq &C(f) \bigg\{\int_{A}+\int_{A^c}\bigg\}\frac{u^{2^*_{\mu}-1}(y)(w^+_{\lambda})(y)u^{2^*_{\mu}-1}(x)(w^+_{\lambda})^2(x)}{|x-y|^{\mu}}dxdy\\
=&:C(f)(K_1+K_2).
\endaligned$$
Since $N-2<\mu<N$, by H\"{o}lder's inequality and HLS inequality we have that
$$\aligned
K_1=&\int_{A}\frac{u^{2^*_{\mu}-1}(y)(w^+_{\lambda})(y)u^{2^*_{\mu}-1}(x)(w^+_{\lambda})^2(x)}{|x-y|^{\mu}}dxdy \\
\leq& \int_{A}\frac{u^{2^*_{\mu}}(y)u^{2^*_{\mu}-2}(x)(w^+_{\lambda})^3(x)}{|x-y|^{\mu}}dxdy \\
\leq& C(\mu,N)\bigg(\int_{\Omega}|u^{2^*_{\mu}}(x)|^{\frac{2^*}{2^*_{\mu}}} dx\bigg)^{\frac{2^*_{\mu}}{2^*}} \bigg(\int_{\Omega}\Big|u^{2^*_{\mu}-2}(y) \big[(w^+_{\lambda})^{\frac{3}{2}}\big]^2\Big|^{ \frac{2^*}{2^*_{\mu}}}\bigg)^{\frac{2^*_{\mu}}{2^*}}\\
\leq&C(\mu,N) ||u||^{2^*_{\mu}}_{L^{2^*}(\Sigma_{\lambda})}
||u||^{2^*_{\mu}-2}_{L^{2^*}(\Sigma_{\lambda})}\Big( \int_{\Sigma_{\lambda}}\big[(w^+_{\lambda})^{\frac{3}{2}}\big]^{2^*}dx\Big)^{\frac{2}{2^*}}\\
\leq&C(\mu,N) ||u||^{2(2^*_{\mu}-1)}_{L^{2^*}(\Sigma_{\lambda})}    \Big(\int_{\Sigma_{\lambda}}\big[(w^+_{\lambda})^{\frac{3}{2}}\big]^{2^*}dx
\Big)^{\frac{2}{2^*}}
\endaligned$$
and
$$\aligned
K_2=&\int_{A^c}\frac{u^{2^*_{\mu}-1}(y)(w^+_{\lambda})(y)u^{2^*_{\mu}-1}(x)(w^+_{\lambda})^2(x)}{|x-y|^{\mu}}dxdy \\
\leq& \int_{A^c}\frac{u^{2^*_{\mu}-3}(y)u^{2^*_{\mu}+1}(x)(w^+_{\lambda})^3(y)}{|x-y|^{\mu}}dxdy \\
\leq&C(\mu,N) \bigg(\int_{A^c}|u^{2^*_{\mu}+1}(x)|^{\frac{2^*}{2^*_{\mu}+1}}dx\bigg)^{\frac{2^*_{\mu}+1}{2^*}} \bigg(\int_{A^c}\Big|u^{2^*_{\mu}-3}(y)\big[(w^+_{\lambda})^{\frac{3}{2}}\big]^2\Big|^{ \frac{2^*}{2^*_{\mu}-1}}dy\bigg)^{\frac{2^*_{\mu}-1}{2^*}}\\
\leq&C(\mu,N) ||u||^{2^*_{\mu}-3}_{L^{2^*}(\Sigma_{\lambda})}
||u||^{2^*_{\mu}+1}_{L^{2^*}(\Sigma_{\lambda})}\Big(\int_{\Sigma_{\lambda}}\big[(w^+_{\lambda})^{\frac{3}{2}}\big]^{2^*}dx\Big)^{\frac{2}{2^*}}\\
\leq&C(\mu,N) ||u||^{2(2^*_{\mu}-1)}_{L^{2^*}(\Sigma_{\lambda})}    \Big(\int_{\Sigma_{\lambda}}\big[(w^+_{\lambda})^{\frac{3}{2}}\big]^{2^*}dx\Big)^{\frac{2}{2^*}}.
\endaligned$$
Since
$$\aligned
\bigg(\int_{\Sigma_{\lambda}}\big[(w^+_{\lambda})^{\frac{3}{2}}\big]^{2^*}dx\bigg)^{\frac{2}{2^*}}
\leq S^2\int_{\Sigma_{\lambda}}\big|\nabla (w^+_{\lambda})^{\frac{3}{2}}\big|^2dx
=\frac{9}{4}S^2\int_{\Sigma_{\lambda}}\big|\nabla w^+_{\lambda}\big|^2w^+_{\lambda}dx,
\endaligned$$
where $S$ is the Sobolev  embedding constant, combining this and the above inequalities, we have
$$\aligned
K_1,K_2\leq S^2C(\mu,N)||u||^{2(2^*_{\mu}-1)}_{L^{2^*}(\Sigma_{\lambda})}
\int_{\Sigma_{\lambda}}|\nabla w^+_{\lambda}|^2w^+_{\lambda}dx.
\endaligned$$
Hence
$$\tilde{I_3}\leq S^2 C(\mu,N,f)||u||^{2(2^*_{\mu}-1)}_{L^{2^*}(\Sigma_{\lambda})}
                \int_{\Sigma_{\lambda}}|\nabla w^+_{\lambda}|^2w^+_{\lambda}dx,$$
where $C(\mu,N,f)>0$ depends only on $\mu,$ $f$ and $N$.
Finally, we focus on the estimate of  $ \tilde{I_4}$. Write  also $A^u_{\lambda}:=\big\{x\in \R^N\setminus R_{\lambda}(\Gamma)\big| u(x)>u_{\lambda}(x)\big\}.$  We deduce from the monotonous properties of $f$, HLS inequality and  H\"{o}lder's inequality that
\begin{equation}\label{eq3.16}
\begin{aligned}
 \tilde{I_4}
 =& \int_{\Omega} \bigg(\int_{\Omega}\Big(\frac{F(u_{\lambda}(y))}{|x-y|^{\mu}} +\frac{F(u(y))}{|x_{\lambda}-y|^{\mu}}\Big)dy\times \big(f(u(x))-f(u_{\lambda}(x))\big)(w^+_{\lambda})^2(x)\varphi^2_{\varepsilon}\varphi^2_R \bigg)dx\\
  \leq& C(f) \int_{\Omega}dx \bigg(\int_{\Omega}\frac{u^{2^*_{\mu}}_{\lambda}(y)+u^{2^*_{\mu}}(y)}{|x-y|^{\mu}}dy \times u^{2^*_{\mu}-2}(x)(w^+_{\lambda})^3(x)\varphi^2_{\varepsilon}\varphi^2_{R}\bigg)\\
 \leq& C(\mu, N)C(f)\Big(||u||^{2^*_{\mu}}_{L^{2^*}(\Sigma_{\lambda})} +||u||^{2^*_{\mu}}_{L^{2^*}(\R^N\setminus \Gamma)}\Big)  \Bigg(\int_{\Omega} \bigg(u^{2^*_{\mu}-2}(x)\big[ (w^+_{\lambda})^{\frac{3}{2}}(x)\big]^2
     \varphi^2_{\varepsilon}\varphi^2_{R}\bigg)^{\frac{2^*}{2^*_{\mu}}} dx\Bigg)^{\frac{2^*_{\mu}}{2^*}} \\
 \leq&  2 C(\mu, N)C(f) ||u||^{2^*_{\mu}}_{L^{2^*}(\R^N\setminus \Gamma)}\bigg(\int_{\Omega}\Big(u^{2^*_{\mu}-2}(x)\big[(w^+_{\lambda})^{\frac{3}{2}}(x)\big]^2
 \varphi^2_{\varepsilon}\varphi^2_R\Big)^{\frac{2^*}{2^*_{\mu}}}dx\bigg)^{ \frac{2^*_{\mu}}{2^*}}\\
 \leq& 2 C(\mu, N)C(f) |u||^{2^*_{\mu}}_{L^{2^*}(\R^N\setminus \Gamma)}\Big(\int_{\Sigma_{\lambda}}u^{2^*}dx\Big)^{ \frac{2^*_{\mu}-2}{2^*}}\Big(\int_{\Sigma_{\lambda}\setminus R_{\lambda}(\Gamma)}\big[(w^+_{\lambda})^{\frac{3}{2}}\big]^{2^*}dx\Big)^{\frac{2}{2^*}} \\
 \leq& 2C(\mu,N)C(f)S^2||u||^{2^*_{\mu}}_{L^{2^*}(\R^N\setminus\Gamma)} ||u||^{2^*_{\mu}-2}_{L^{2^*}(\Sigma_{\lambda})}\Big(\int_{\Sigma_{\lambda}} \big|\nabla(w^+_{\lambda})^{ \frac{3}{2}}\big|^{2}dx\Big)\\
 \leq&  2\times{\frac{9}{4}} C(\mu, N)C(f)S^2 ||u||^{2^*_{\mu}}_{L^{2^*}(\R^N\setminus\Gamma)}||u||^{2^*_{\mu}-2}_{L^{2^*}( \Sigma_{\lambda})} \Big(\int_{\Sigma_{\lambda}}\big|\nabla(w^+_{\lambda})\big|^{2}(w^+_{\lambda})dx\Big) \\
 \leq&  C(\mu, N,f)S^2
 ||u||^{2^*_{\mu}}_{L^{2^*}(\R^N\setminus\Gamma)}||u||^{2^*_{\mu}-2}_{L^{2^*}(\Sigma_{\lambda})}\Big(\int_{\Sigma_{\lambda}}
        \big|\nabla(w^+_{\lambda})\big|^{2}w^+_{\lambda}dx\Big),
 \end{aligned}
\end{equation}
where $S$ is the Sobolev imbedding constant from $\mathcal{D}^{1,2}(\mathbb{R}^N)$ to $L^{2^*}(\mathbb{R}^N)$. It follows from Lemma \ref{Lem3.1} and \eqref{eq3.9} that the integral $||u||_{L^{2^*}(\Sigma_{\lambda}\cap(B_{2R}\setminus B_R))}\to 0$ as $R\to+\infty$ or $\lambda\to-\infty$.
Inserting $($\ref{eq3.13}$)$ - $($\ref{eq3.16}$)$ into $($\ref{eq3.12}$)$ and let $\varepsilon\to 0^+$ and $R\to +\infty$, respectively. It follows from \eqref{eq3.9} and the Lebesgue dominated convergence theorem that
$$\aligned
2\int_{\Sigma_{\lambda}}\big|\nabla w^+_{\lambda}\big|^2 w^+_{\lambda} dx
\leq&\liminf\limits_{(\varepsilon,R)\to(0,+\infty)}2\int_{\Sigma_{\lambda}}\big|\nabla w^+_{\lambda}\big|^2 w^+_{\lambda}\varphi^2_{\varepsilon}\varphi^2_R dx\\
\leq&\liminf\limits_{(\varepsilon,R)\to(0,+\infty)}\big\{ |\tilde{ I_1}|+|\tilde{I_2}|+\tilde{I_3}+\tilde{I_4}\big\} \\
\leq& 2 S^2C(\mu, N,f)
||u||^{2^*_{\mu}}_{L^{2^*}(\R^N\setminus\Gamma)}
||u||^{2^*_{\mu}-2}_{L^{2^*}(\Sigma_{\lambda})}\Big(\int_{\Sigma_{\lambda}}
        \big|\nabla(w^+_{\lambda})\big|^{2}w^+_{\lambda}dx\Big) .
\endaligned$$
Thus,
\begin{equation}\label{eq3.17}
\int_{\Sigma_{\lambda}} |\nabla w_{\lambda}^+|^2 w_{\lambda}^+ dx
\leq  S^2C(\mu, N,f)
 ||u||^{2^*_{\mu}}_{L^{2^*}(\R^N\setminus\Gamma)}
 ||u||^{2^*_{\mu}-2}_{L^{2^*}(\Sigma_{\lambda})}\Big(\int_{\Sigma_{\lambda}}
        \big|\nabla(w^+_{\lambda})\big|^{2}w^+_{\lambda}dx\Big)  ,
\end{equation}
Observe that $u\in L^{2^*}(\Sigma_{\lambda})$ implies that $||u||_{L^{2^*}(\Sigma_{\lambda})}\to0$ as $\lambda\to -\infty.$ As a consequence, there is a sufficiently large positive constant $M$ such that, for all $\lambda<-M$,
$$   S^2C(\mu, N,f)
 ||u||^{2^*_{\mu}}_{L^{2^*}(\R^N\setminus\Gamma)}
 ||u||^{2^*_{\mu}-2}_{L^{2^*}(\Sigma_{\lambda})} <1.$$
Thereby, $\int_{\Sigma_{\lambda}}\big|\nabla w^+_{\lambda}\big|^2w^+_{\lambda}dx=0,$ which means that $w^+_{\lambda}=const.$ Since $w^+_{\lambda}$ is continuous on $\Sigma_{\lambda}\cup T_{\lambda}$ as well as zero on $T_{\lambda},$ we get $w^+_{\lambda}=0,$ and the claim is proved.
\par
In order to use the moving plane with $T_{\lambda}$ and obtain the monotonicity result of $u$ as desired, we define
$$\aligned
\Lambda_0:=\Big\{\lambda<0: \, u\leq u_{t} \,\, in\,\, \Sigma_{t}\setminus R_t(\Gamma) ,\,\, \forall t \in (-\infty,\lambda]\Big\}
\endaligned$$
and  $$\lambda_0:=\sup \Lambda_0.$$

\textbf{Step2:} We demonstrate that $\lambda_0=0$.
Arguing by contradiction, we suppose that  $\lambda_0<0$. By the continuity, we have
$$u(x)\leq u_{\lambda_0}(x), \quad \forall x\in \Sigma_{\lambda_0}\setminus R_{\lambda_0}(\Gamma).$$
If not, then there exists a point $\bar{x}\in \Sigma_{\lambda_0}\setminus R_{\lambda_0}(\Gamma)$ so that $u(\bar{x})>u_{\lambda_0}(\bar{x})$. Due to the compactness of singular set $R_{\lambda_0}(\Gamma)$, it follows that $\Sigma_{\lambda_0}\setminus R_{\lambda_0}(\Gamma)$ is open. Let
$$g(x,\lambda):=u(x)-u_{\lambda}(x)=u(x)-u(x_{\lambda}), \quad (x,\lambda)\in (\Sigma_{\lambda_0}\setminus R_{\lambda_0}(\Gamma))\times(-\infty,\lambda_0],$$
then $g$ is continuous. So there is a neighborhood $U(\bar{x})\subset \Sigma_{\lambda_0}\setminus R_{\lambda_0}(\Gamma)$  of the point $\bar{x}$  and a real number $\delta>0$   such that
$$g(x,t)>g(\bar{x},\lambda_0)=u(\bar{x})-u_{\lambda_0}(\bar{x})>0, \quad \forall (x,t)\in U(\bar{x})\times (\lambda_0-\delta,\lambda_0).$$
In particular, we get $u(\bar{x})>u_{t}(\bar{x})$ for any $\lambda_0-\delta<t<\lambda_0.$ In view of the definition of $\Lambda_0,$ we see there must be some $\tilde{\lambda}\leq \lambda_0-\delta$ satisfying that $T_{\tilde{\lambda}}$ is the limit position of the moving plane, i.e. we have
 $$\sup \Lambda_0 \leq \tilde{\lambda}<\lambda_0,$$ which makes a contradiction with the definition of the supremum. Hence the strong comparison principle (see Theorem \ref{Thm4.1} in the Appendix) implies that $u<u_{\lambda_0}$ in $\Sigma_{\lambda_0}\setminus R_{\lambda_0}(\Gamma).$ Indeed, it is impossible that $u=u_{\lambda_0}$ in $\Sigma_{\lambda_0}\setminus R_{\lambda_0}(\Gamma)$ provided $\lambda_0<0,$ since in this case $u$ is singular somewhere on $\Sigma_{\lambda_0}\setminus R_{\lambda_0}(\Gamma).$ Now, for the constant $\bar{\delta}>0$ which is going to be determined later, we claim that
$$u\leq u_{\lambda_0+\delta} \quad \text{in} \quad  \Sigma_{\lambda_0+\delta}\setminus R_{\lambda_0+\delta}(\Gamma)$$
for any $\delta\in (0,\bar{\delta}).$ As a result, $\lambda_0+\delta \in\Lambda_0,$ which contradicts with the definition of $\lambda_0.$ And whence we deduce that the hypothesis of $\lambda_0<0$ is not true, so it must be $\lambda_0=0.$ We are left with the proof of the claim above. It is easy to see that, for arbitrary $\varepsilon>0,$ there exist $\delta_1=\delta_1(\varepsilon, \lambda_0)>0$ and a suitable compact set $K$(only depending on $\varepsilon$ and $\lambda_0$) such that
$$\int_{\Sigma_{\lambda_0}\setminus K}u^{2^*}dx<\frac{\varepsilon}{2}, \quad \forall \lambda\in [\lambda_0,\lambda_0+\delta_1].$$
In addition, set
$$g(x,\lambda):=u(x)-u_{\lambda}(x)=u(x)-u(2\lambda-x_1,x')\in C(K\times [\lambda_0,\lambda_0+\delta_1]),$$
$g$ is continuous on the compact set $K\times [\lambda_0,\lambda_0+\delta_1]$. Hence there is a sufficiently small $\delta_2\in[0,\delta_1)$ such that
$$|g(x,\lambda_0)-g(x,\lambda_0+t)|\leq -\frac{g(x,\lambda_0)}{2}$$
for all $(x,t)\in [0,\delta_2)\times K,$
that is,
$$g(x,\lambda_0+t)\leq g(x,\lambda_0)-\frac{g(x,\lambda_0)}{2}=\frac{g(x,\lambda_0)}{2}=\frac{1}{2}(u(x)-u_{\lambda_0}(x))<0,$$
which implies that
$$u(x)<u_{\lambda_0+t}(x), \quad \forall x\in K, \,\forall t\in[0,\delta_2). $$
Because of the fact that $u\in L^{2^*}(\Sigma_{\lambda})$ for any $\lambda<0,$ we can take $\delta_2<\frac{|\lambda_0|}{4}$ with $\lambda_0+\frac{|\lambda_0|}{4}=\frac{3\lambda_0}{4}<0.$
By $u\in L^{2^*}(\Sigma_{\lambda_0+\frac{|\lambda_0|}{4}})$,  we get that there exists $\bar{\delta}\in(0,\delta_2)$ such that
$$\aligned
\int_{\Sigma_{\lambda}\setminus \Sigma_{\lambda_0}}u^{2^*}dx\leq \int_{\Sigma_{\lambda_0+\delta_2}\setminus \Sigma_{\lambda_0}}u^{2^*}dx<\frac{\varepsilon}{2},
\endaligned$$
for any $\lambda \in [\lambda_0,\lambda_0+\bar{\delta}]$.
Thus by combining $\int_{\Sigma_{\lambda_0}\setminus K}u^{2^*}dx<\frac{\varepsilon}{2}$, we have
$$\aligned
\int_{\Sigma_{\lambda}\setminus K}u^{2^*}dx=&\Big\{\int_{\Sigma_{\lambda_0}\setminus K}+\int_{\Sigma_{\lambda}\setminus \Sigma_{\lambda_0}}\Big\}u^{2^*}dx<\varepsilon
\endaligned$$
for all $\lambda\in[\lambda_0,\lambda_0+\bar{\delta}].$
Now we substitute the original test function with
$$\varphi:=w^+_{\lambda_0+\delta}\varphi^2_{\varepsilon}\varphi^2_{R} \chi_{\Sigma_{\lambda_0+\delta}}.$$
Repeat the argument in the step $1$,  by the $($\ref{eq3.17}$)$ we deduce that,  for any $0\leq\delta<\bar{\delta}$,
\begin{equation}\label{eq3.18}
\begin{aligned}
\int_{\Sigma_{\lambda_0+\delta}\setminus K} |\nabla w_{\lambda_0+\delta}^+|^2 w_{\lambda_0+\delta}^+ dx
\leq&  S^2C(\mu, N,f)
 ||u||^{2^*_{\mu}}_{L^{2^*}(\R^N\setminus\Gamma)}
 ||u||^{2^*_{\mu}-2}_{L^{2^*}(\Sigma_{\lambda_0+\delta}\setminus K)}\\
 &\times \Big(\int_{\Sigma_{\lambda_0+\delta}\setminus K}
        \big|\nabla(w^+_{\lambda_0+\delta})\big|^{2}w^+_{\lambda_0+\delta}dx\Big) .
 \end{aligned}
\end{equation}
According to  the construction above, it follows that $g(x,\lambda_0+\delta)<0$ for any $x\in K$. Therefore, the continuity of  $g$ yields that $w^+_{\lambda_0+\delta}=0=\nabla w^+_{\lambda_0+\delta}$ in a neighborhood of $K$. For fixed $\varepsilon<\Big[4S\cdot C(\mu,N)||u||^{2^*_{\mu}}_{L^{2^*}(\R^N\setminus \Gamma)}\Big]^{-\frac{2}{2^*_{\mu}-2}},$ then it holds
$$\aligned
&S\cdot C(\mu, N) ||u||^{2^*_{\mu}}_{L^{2^*}(\R^N\setminus \Gamma)}\Big(\int_{\Sigma_{\lambda_0+\delta}\setminus K}u^{2^*}dx\Big)^{\frac{2^*_{\mu}-2}{2}}
\leq 2S\cdot C(\mu, N) ||u||^{2^*_{\mu}}_{L^{2^*}(\R^N\setminus \Gamma)}\varepsilon^{\frac{2^*_{\mu}-2}{2}}<\frac{1}{2}, \forall 0\leq\delta<\bar{\delta}.
\endaligned$$
Plugging this into $($\ref{eq3.18}$)$, it leads to
$$\int_{\Sigma_{\lambda_0+\delta}\setminus K}|\nabla w^+_{\lambda_0+\delta}|^2w^+_{\lambda_0+\delta}dx=0,\quad \forall 0\leq\delta<\bar{\delta}.$$
Since $\nabla w^+_{\lambda_0+\delta}$ is zero in a neighborhood of $K,$ we have
$$\int_{\Sigma_{\lambda_0+\delta}}|\nabla w^+_{\lambda_0+\delta}|^2w^+_{\lambda_0+\delta}dx=0, \forall \delta\in[0,\bar{\delta}).$$
Note that the restriction of $w^+_{\lambda_0}$ on $T_{\lambda_0+\delta}$ is zero, we conclude that $w^+_{\lambda_0+\delta}\equiv 0$ in $\Sigma_{\lambda_0+\delta}\setminus R_{\lambda_0+\delta}(\Gamma)$. As a consequence, $u\leq u_{\lambda+0+\delta}$ in $\Sigma_{\lambda_0+\delta}\setminus R_{\lambda_0+\delta}(\Gamma)$ for all $\delta\in[0,\bar{\delta})$. Thus the step $2$ is completed.
\par
\textbf{Step3.}
Similar to above steps, we can perform the moving plane procedure in the opposite direction and then obtain the symmetry of the positive singular solution.
Moreover, the moving plane technique implies that the positive singular solution is increasing in the $x_1$- direction in $\{x\in \R^N: x_1 <0\}$, i.e.,
$$u_{x_1}>0 ~~\text{in} ~~ \{x_1 <0\}.$$
If the solution $u$ has only a nonremovable singularity at the origin $\{0\}$, then the solution $u$ is radial and radially decreasing about $\{0\}$ by the moving plane procedure. This accomplishes the proof of Theorem \ref{Thm1.1}.
\raisebox{-0.5mm}{\rule{2.5mm}{3mm}}\vspace{6pt}
\par

\section{A model problem}  \label{sec:proof2}
Let us first introduce the Riesz potential operator
$$I_{\alpha}f(x)=\frac{1}{\gamma(\alpha)}\int_{\R^N}\frac{f(y)}{|x-y|^{N-\alpha}}dy,$$
where $0<\alpha<N$, $\gamma({\alpha})= {2^{\alpha}\pi^{{N}/{2}}\Gamma(\frac{\alpha}{2}) }/{\Gamma(\frac{N-\alpha}{2})}$ and $\Gamma(z)=\int_0^{+\infty}t^{z-1}e^{-t}dt$.

We state some useful results as follows:
\begin{Thm}\label{Thm3.1}
(\cite[Theorem 1.33]{Maly-William1997AMSPR}) Let $\alpha>0$, $1<p<\infty$, and $\alpha p<N$. Then, there exists a positive constant $C=C(N,p,\alpha)$ such that
$$||I_{\alpha} f||_{q} \leq C||f||_p, \qquad q=\frac{Np}{N-\alpha p},$$
whenever $f\in L^p(\R^N)$.
\end{Thm}

\begin{Thm}(\cite{Stein1993Book})
Let $\mathcal{F}(u)$ denotes the Fourier transform of the function $u$ by
$$\mathcal{F}(u)(x)=\int_{\R^N}u(\xi)e^{-2\pi \sqrt{-1}x\cdot \xi}d\xi.$$
Let $0<\alpha<N$, we have
\begin{itemize}\label{Thm3.2}
\item[(\romannumeral1)]$\mathcal{F}^{-1}\big(\frac{1}{|x|^{N-\alpha}}\big)=\mathcal{F}\big(\frac{1}{|x|^{N-\alpha}}\big)=\gamma(\alpha)\frac{1}{(2\pi |x|)^{\alpha}}$.
\item[(\romannumeral2)] In addition, it holds that
$$\mathcal{F}(I_{\alpha}f)=\frac{1}{(2\pi |x|)^\alpha}\mathcal{F}(f)(x)$$
\end{itemize}
where $f$ belongs to the Schwartz space $\mathcal{S}(\R^n)$.
\end{Thm}

\par
$\mathbf{Proof~of~Theorem~\ref{Thm1.2}.}$~~
By using the Riesz potential operator and Theorem \ref{Thm3.2}, \eqref{eq1.8}  is equivalent to the following equation
\begin{equation}\label{eq}
u^{-q}(-\Delta u)=\gamma(N-\mu) I_{N-\mu}(u^p).
\end{equation}
Set $u(x)=\frac{A}{|x|^s}$,  we next determine $A$ and $s$. Direct computations yield
$$\aligned
u^{-q}(-\Delta u)=A^{1-q}|x|^{sq}(-\Delta(|x|^{-s})) =A^{1-q}s(N-2-s)|x|^{s(q-1)-2},
\endaligned$$
and
$$\aligned
\gamma(N-\mu)I_{N-\mu}(u^p)&=\gamma(N -\mu)I_{N-\mu}\big({(\frac{A}{|x|^s})}^p\big)=A^p\gamma(N-\mu)I_{N-\mu}(|x|^{-sp}).
\endaligned$$
Combining the above two identities and \eqref{eq}, we have
\begin{equation*}
s(N-2-s)|x|^{s(q-1)-2}=A^{p+q-1}\gamma(N-\mu)I_{N-\mu}(|x|^{-sp}),
\end{equation*}
It follows from the Fourier transform and the Theorem \ref{Thm3.2} that
\begin{equation*}
s(N-2-s)\mathcal{F}(|x|^{s(q-1)-2})=A^{p+q-1}\gamma(N-\mu)\mathcal{F}(I_{N-\mu}(|x|^{-sp}))
=A^{p+q-1}\gamma(N-\mu)\frac{\mathcal{F}(|x|^{-sp})}{(2\pi |x|)^{N-\mu}},
\end{equation*}
which means
$$\frac{s(N-2-s)\gamma(N-2+s(q-1))}{(2\pi)^{N-2+s(q-1)}|x|^{N-2+s(q-1)}}
=A^{p+q-1}\frac{\gamma(N-\mu)\gamma(N-sp)}{(2\pi)^{2N-\mu-sp}|x|^{2N-\mu-sp}},$$
where the parameters need to satisfy
$$\left\{
\begin{array}{l}
 0<\mu<N, \\0<N-2+s(q-1)<N, \\0<sp<N
\end{array}
\right.$$
and
$$\left\{
\begin{array}{l}
 N-2+s(q-1)=2N-\mu-sp,\\
 s(N-2-s)\gamma(N-2+s(q-1))=A^{p+q-1}\gamma(N-\mu)\gamma(N-sp).
\end{array}
\right.$$
Thus, we have that
$$s=\frac{N-\mu+2}{p+q-1}>0,$$
and
$$\aligned
A^{p+q-1}=\frac{s(N-2-s)\gamma(N-2+s(q-1))}{\gamma(N-\mu)\gamma(N-sp)}  >0.  \endaligned$$
\raisebox{-0.5mm}{\rule{2.5mm}{3mm}}\vspace{6pt}
\par

\section{Appendix}  \label{sec:app}
\begin{Thm}\label{Thm4.1}
(Strong Comparison Principle) Let $\Sigma^{\prime}_{\lambda}:=\Sigma_{\lambda}\setminus R_{\lambda}(\Gamma)$ be a $N$-dimensional connected open set, where $\Sigma_{\lambda}$ is defined as before. Suppose that $u$ and $u_{\lambda}$ are continuous on $\Sigma^{\prime}_{\lambda}$ with
\begin{equation}\label{eq4.1}
\begin{aligned}
&\int_{\Sigma_{\lambda}}\nabla u_{\lambda}\nabla\varphi dx- \int_{\Sigma_{\lambda}}\Big(\int_{\R^N\setminus R_{\lambda}(\Gamma)}
      \frac{F(u_{\lambda})(y)}{|x-y|^{\mu}}dy\Big)f(u_{\lambda}(x))\varphi(x)dx\geq0, \\
&\int_{\Sigma_{\lambda}}\nabla u\nabla\varphi dx- \int_{\Sigma_{\lambda}}\Big(\int_{\R^N\setminus \Gamma}
      \frac{F(u(y))}{|x-y|^{\mu}}dy\Big)f(u(x))\varphi(x)dx\leq0\\
\end{aligned}
\end{equation}
for any nonnegative $\varphi\in C^1_c(\R^N\setminus R_{\lambda}(\Gamma)).$ Suppose that $u\leq u_{\lambda}$ in $\Sigma^{\prime}_{\lambda},$ then either $u\equiv u_{\lambda}$ in $\Sigma^{\prime}_{\lambda}$ or $u<u_{\lambda}$ in $\Sigma^{\prime}_{\lambda}.$
\end{Thm}
\begin{proof}
If $u<u_{\lambda}$ in $\Sigma^{\prime}_{\lambda}$, the theorem holds. Next, suppose there exists  a point $x_0\in\Sigma^{\prime}_{\lambda}$ such that $u(x_0)=u_{\lambda}(x_0).$ Since $u\leq u_{\lambda}$ in $\Sigma^{\prime}_{\lambda},$ $x_0$ is a minimum point of $u_{\lambda}-u$ in $\Sigma^{\prime}_{\lambda}.$ Set $$\Omega^{\prime}=\{x\in \Sigma^{\prime}_{\lambda}| u(x)=u_{\lambda}(x)\},$$
then,  $\Omega^{\prime}$ is relatively closed to $\Sigma^{\prime}_{\lambda}$ due to the continuity of $u$ and $u_{\lambda}.$ It suffices to show that $\Omega^{\prime}$ is open. Indeed, choosing a ball centered at $x_0$ with radius $\rho$ small enough such that $B_{4\rho}(x_0)\subset \Sigma^{\prime}_{\lambda}.$ And let $\omega$ be the restriction of $u_{\lambda}-u$ on $B_{4\rho}(x_0)$ and vanish outside $B_{4\rho}(x_0).$ On the other hand, for any nonnegative $\phi\in C^1_{c}(B_{4\rho}(x_0)),$ it follows from $($\ref{eq4.1}$)$ that
\begin{equation*}
\begin{aligned}
\int_{B_{4\rho}(x_0)}\nabla u_{\lambda}\nabla\varphi dx - \int_{B_{4\rho}(x_0)}\Big(\int_{\R^N\setminus R_{\lambda}(\Gamma)}
        \frac{F(u_{\lambda}(y))}{|x-y|^{\mu}}dy\Big) f(u_{\lambda}(x))\varphi(x)dx \geq 0
\end{aligned}
\end{equation*}
and
\begin{equation*}
\begin{aligned}
\int_{B_{4\rho}(x_0)}\nabla u\nabla\varphi dx - \int_{B_{4\rho}(x_0)}\Big(\int_{\R^N\setminus R_{\lambda}(\Gamma)}
        \frac{F(u(y))}{|x-y|^{\mu}}dy\Big) f(u(x))\varphi(x)dx \leq 0.
\end{aligned}
\end{equation*}
We deduce from the monotonicity of $f$  and the two above inequalities  that
\begin{equation*}
\begin{aligned}
\int_{B_{4\rho}(x_0)} \nabla \omega \nabla\varphi dx
& \geq \int_{B_{4\rho}(x_0)} \Big[\int_{\R^N\setminus R_{\lambda}(\Gamma)}
        \frac{F(u_{\lambda}(y))}{|x-y|^{\mu}}dyf(u_{\lambda}(x))-\int_{\R^N\setminus \Gamma}
        \frac{F(u(y))}{|x-y|^{\mu}}dyf(u(x))\Big]\varphi(x)dx\\
&\quad=\int_{B_{4\rho}(x_0)} \bigg(\int_{\Sigma_{\lambda}\setminus R_{\lambda}(\Gamma)} \big(F(u(y))-F(u_{\lambda}(y))\big)
        \Big(\frac{1}{|x_{\lambda}-y|^{\mu}}-\frac{1}{|x-y|^{\mu}}\Big)dy\bigg)\varphi(x)dx\\
&\quad+\int_{B_{4\rho}(x_0)}\Big(\int_{\R^N\setminus \Gamma}\frac{F(u(y))}{|x-y|^{\mu}}\Big)(f(u_{\lambda}(x))-f(u(x)))\varphi(x)dx\\
&\geq 0,
\end{aligned}
\end{equation*}
which implies that $\omega$ is a nonnegative weak supersolution of
\begin{equation}
-\Delta \omega=0  \quad in \quad B_{4\rho}(x_0).
\end{equation}
By means of the weak Harnack inequality \cite[Lemma 2.113]{Maly-William1997AMSPR}, we have
$$\aligned
0\leq \crossint_{B_{\rho}(x_0)} \omega(x)\,dx \leq C(N)\inf\limits_{B_{\rho}(x_0)}\omega=\omega(x_0)=0,
\endaligned$$
which implies $\omega\equiv 0$ in $B_{\rho}(x_0)\subset\Sigma^{\prime}_{\lambda}.$ So $B_{\rho}(x_0)\subset \Omega^{\prime},$ i.e. $\Omega^{\prime}$ is open relatively to $\Sigma^{\prime}_{\lambda}.$ This completes the proof.
\end{proof}

\section*{Acknowledgments}
This work was supported by the National Natural Science Foundation of China (12261107, 12101546), Yunnan Fundamental Research Projects (202301AU070144, 202301AU070159), Scientific Research Fund of Yunnan Educational Commission (2023J0199, 2023Y0515), and Yunnan Key Laboratory of Modern Analytical Mathematics and Applications (202302AN360007). A.J. is partially supported by PRIN 2022 "\emph{Pattern formation in nonlinear phenomena}" and is a member of the INDAM Research Group GNAMPA.

\end{document}